\documentclass[11pt, oneside]{amsart}   	
\usepackage{amsmath,amsthm,amscd,amsfonts, amssymb, mathrsfs}
\usepackage{geometry}                		
\usepackage{graphicx}				
\usepackage{amssymb}
\newcommand{\sect}[1]{\section{#1}\setcounter{equation}{0}}
\newcommand{\subsect}[1]{\subsection{#1}}

\font\mbn=msbm10 scaled \magstep1
\font\mbs=msbm7 scaled \magstep1
\font\mbss=msbm5 scaled \magstep1
\newfam\mbff
\textfont\mbff=\mbn
\scriptfont\mbff=\mbs
\scriptscriptfont\mbff=\mbss   

\newcommand{\Di}      {\mathbb{D}}

\newcommand{\N}       { \mathbb{N}}
\newcommand{\Z}        {\mathbb{Z}  }  
\newcommand\Co           {{\mathbb C}}

\newtheorem{Th}{Theorem}[section]
\newtheorem{Lm}[Th]{Lemma}
\newtheorem{C}[Th]{Corollary}

\newtheorem{Prop}[Th]{Proposition}

\newtheorem{E}[Th]{Example}

\newtheorem*{Qu}{Question}
\newtheorem*{Th A}{Theorem A}
\newtheorem*{Th B}{Theorem B}
\begin{document}

\title[Projective Freeness of  Algebras of Bounded Holomorphic Functions ]{Projective Freeness of Algebras of Bounded Holomorphic Functions  on Infinitely Connected Domains}
\author{Alexander Brudnyi}
\address{Department of Mathematics and Statistics\newline
\hspace*{1em} University of Calgary\newline
\hspace*{1em} Calgary, Alberta, Canada\newline
\hspace*{1em} T2N 1N4}
\email{abrudnyi@ucalgary.ca}


\thanks{Research is supported in part by NSERC}

\begin{abstract}
The algebra $H^\infty(D)$ of bounded holomorphic functions on $D\subset\mathbb C$ is projective free for a wide class of infinitely connected domains. In particular, for such $D$ every rectangular left-invertible matrix with entries in $H^\infty(D)$ can be extended in this class of matrices to an invertible square matrix (the generalization of the corona theorem for $H^\infty(D)$). This follows from a new result on the structure of the maximal ideal space of $H^\infty(D)$ asserting that its covering dimension is $2$ and the second \v{C}ech cohomology group is trivial. 
\end{abstract}

\date{}

\maketitle
\sect{Formulation of the Main Results}
\subsect{} Let $H^\infty(X)$ be the Banach algebra of bounded holomorphic functions on a complex manifold $X$ equipped with pointwise multiplication and supremum norm and let
$\mathfrak M(H^\infty(X))$ be the {\em maximal ideal space} of $H^\infty(X)$, i.e., the set of nonzero homomorphisms $H^\infty(X)\rightarrow\Co$ endowed with the weak$^*$ topology of the dual space $H^\infty(X)^*$ (the {\em Gelfand topology}). It is  a compact Hausdorff space contained in the closed unit ball of $H^\infty(X)^*$. If $X$ is {\em Caratheodory hyperbolic} (i.e., $H^\infty(X)$ separates the points of $X$),  it  can be identified with an open subset of $\mathfrak M(H^\infty(X))$ formed by evaluation functionals at points of $X$. Then the {\em corona problem}  asks whether $X$ is dense in $\mathfrak M(H^\infty(X))$. The most famous corona problem for $H^\infty$ on the unit disk $\Di$ was posed by Kakutani in 1941 and solved positively by Carleson \cite{C} in 1962 (see the book \cite{DKSTW} and references therein for other significant results in this area).
The general corona problem for algebras of bounded holomorphic functions on plane domains is still open as is the problem in several variables for the ball and polydisk.

Denseness of $X$ in $\mathfrak M(H^\infty(X))$ can be equivalently reformulated as follows, see, e.g., \cite[Ch.\,V]{Ga}: 

For every family   $f_1,\dots, f_n\in H^\infty(X)$, $n\in\N$, satisfying the {\em corona condition}
\begin{equation}\label{eq1.2}
\sup_{x\in X}\max_{1\le i\le n}|f_i(x)|>0
\end{equation}
there exist $g_1,\dots, g_n\in H^\infty(X)$ such that 
\begin{equation}\label{eq1.3}
\sum_{i=1}^n\, f_i\!\cdot\! g_i=1.
\end{equation}

A more general problem on extension of matrices over $H^\infty(X)$ is as follows:\smallskip

Let $A$ be a $k\times n$ matrix, $k<n$, with entries in $H^\infty(X)$ such that the family of $k$-minors satisfies the corona condition. Is there a $n\times n$ matrix $B$ with entries in $H^\infty(X)$ and ${\rm det}(B)=1$ extending $A$ (i.e., containing $A$ as a submatrix)?

If this is true for all such matrices $A$ with $k,n\in\N$, then $H^\infty(X)$ is called {\em Hermite}. Equivalently, $H^\infty(X)$ is Hermite if every finitely generated {\em stably free} $H^\infty(X)$-module is free.

It was first shown by Tolokonnikov that $H^\infty(X)$ is Hermite for $X$ being a Caratheodory hyperbolic  Riemann surface of finite type \cite[Thm.\,3]{To1} or a plain domain obtained by deleting from $\mathbb D^*:=\Di\setminus\{0\}$ a hyperbolically-rare sequence of closed disks converging to $\{0\}$, \cite[Cor.]{To2}. Later in \cite{Br2}, \cite{Br4} the author proved the following generalizations of  \cite[Thm.\,3]{To1}. 

Let $N$ be a bordered Riemann surface. A connected Riemann surface $U$ is called an {\em $N$-domain} if there is a holomorphic embedding of $U$ in an unbranched covering of $N$ inducing monomorphism from $\pi_1(U)$ to the fundamental group of the covering. 

Let $\mathscr U=\{U_\alpha\}_{\alpha\in\Lambda}$ be a family of (not necessarily distinct) $N$-domains and $X_{\mathscr U}:=\sqcup_{\alpha\in\Lambda} U_\alpha$. Then the Banach algebra $H^\infty(X_{\mathscr U})$ is Hermite, see \cite[Thm.\,1.1]{Br2}. 

This result was used by the author to prove a similar result for $H^\infty$ on unbranched coverings of Riemann surfaces of finite type. Specifically, let $N$ be a Caratheodory hyperbolic Riemann surface of finite type. Let $\mathscr U=\{U_\alpha\}_{\alpha\in\Lambda}$ be a family of (not necessarily distinct) unbranched coverings of $N$  and $X_{\mathscr U}:=\sqcup_{\alpha\in\Lambda} U_\alpha$. Then
the Banach algebra $H^\infty(X_{\mathscr U})$ is Hermite, see \cite[Thm.\,1.2]{Br4}.\smallskip

The concept of a Hermite ring is a weaker notion than that of a projective free ring. A commutative ring with identity $R$ is said to
be {\em projective free} if every finitely generated projective
$R$-module is free (i.e., if $M$ is an $R$-module such that
$M \oplus N\cong R^{d}$ for an $R$-module $N$ and $d\in\Z_+$, then
$M\cong R^{k}$ for some $k\in\Z_+$).  Since every stably free module is projective, every projective free ring is Hermite.
In terms of matrices, $R$
is projective free iff every nontrivial square idempotent matrix over $R$ is similar
(by an invertible matrix over $R$) to a matrix of the form
$$
\textrm{diag}(I_m,0):=\left[ \begin{array}{cc} I_m & 0 \\ 0 & 0 
\end{array} \right],\quad m\in\N,
$$
where $I_m$ is the $m\times m$ identity matrix, see, e.g., \cite[Prop.\,2.6]{Co}.

If $A$ is a projective free complex commutative unital Banach algebra, then its maximal ideal space $\mathfrak M(A)$ is connected, and the \v{C}ech cohomology group $H^2(\mathfrak M(A),\Z)=0$, see, e.g., \cite[Sect.\,3]{Br5} for the  corresponding references.

In \cite[Thm.\,1.5]{BS} the authors proved that $H^\infty(U)$ is projective free for an $N$-domain $U$. The proof relies on the Lax--Halmos-type theorem \cite[Thm.\,1.7]{Br2}. In fact, using uniform estimates of that theorem and arguing as in the proof of \cite[Thm.\,1.5]{BS} one obtains a more general statement generalizing \cite[Thm\,1.1]{Br2}. Specifically, let $\mathscr U=\{U_\alpha\}_{\alpha\in\Lambda}$ be a family of (not necessarily distinct) $N$-domains and $X_{\mathscr U}:=\sqcup_{\alpha\in\Lambda} U_\alpha$. Let $F$ be a nontrivial matrix idempotent on $X_{\mathscr U}$ with entries in $H^\infty(X_{\mathscr U})$ such that $F(X_{\mathscr U})$ is connected. Then $F$ is similar (by an invertible matrix over $H^\infty(X_{\mathscr U})$) to a matrix $\textrm{diag}(I_m,0)$, $m\in\N$.
\subsect{} In the present paper, continuing this line of research, we prove projective freeness of $H^\infty$ on a new wide class of plain domains, the so-called $\mathscr B$-domains, introduced and studied by Behrens \cite{Be}. By definition, $U$ is a $\mathscr B$-domain if it is obtained from a domain $V\subset\Co$ by
deleting a (possibly finite) {\em hyperbolically-rare} sequence of closed disks $\{\Delta_n\}\subset V$ with centers $\alpha_n$, i.e., such that there are disjoint closed disks $D_n$ with centers $\alpha_n$ satisfying $\Delta_n\Subset D_n\subset V$ and 
\begin{equation}\label{equ1.1}
\sum\, \frac{{\rm rad}(\Delta_n)}{{\rm rad}(D_n)}<\infty.
\end{equation}
Behrens \cite[Thm.\,6.1]{Be} proved that if a $\mathscr B$-domain $U$ is constructed from a domain $V$ for which the corona theorem is valid (i.e., $V$ is dense in $\mathfrak M(H^\infty(V))$), then it is valid for  $U$ as well. In this case, he also described the topological structure of $\mathfrak M(H^\infty(U))$. 

The main result of our paper reveals some additional topological properties of $\mathfrak M(H^\infty(U))$. For its formulation, recall that for a normal space $X$, ${\rm dim}\,X \le n$ if every open cover of $X$ can be refined by an open cover whose order $\le n + 1$. If  ${\rm dim}\,X \le n$ and the statement  ${\rm dim}\,X \le n-1$ is false, then ${\rm dim}\,X = n$.
\begin{Th}\label{te1.1}
Suppose $U$ is obtained from a domain $V\subset\mathbb C$ by deleting a (possibly finite) hyperbolically-rare sequence of closed disks such that
\begin{itemize}
\item[(i)] $V$ is dense in $\mathfrak M(H^\infty(V))$;\smallskip
\item[(ii)] The covering dimension ${\rm dim}\, \mathfrak M(H^\infty(V)) =2$;\smallskip
\item[(iii)] The \v{C}ech cohomology group $H^2(\mathfrak M(H^\infty(V)),\Z)=0$.\smallskip 
\end{itemize}
Then
${\rm dim}\, \mathfrak M(H^\infty(U)) =2$ and $H^2(\mathfrak M(H^\infty(U)),\Z)=0$ as well.
\end{Th}

As a direct corollary of Theorem \ref{te1.1} we obtain the following.
\begin{C}\label{cor1.2}
Under conditions of the theorem the Banach algebras $H^\infty(V)$ and $H^\infty(U)$ are projective free.
\end{C}
For instance, taking here $V:=\Di^*$ and $U$ constructed from $V$ by deleting a  hyperbolically-rare sequence of closed disks converging to $\{0\}$ we obtain the generalization of \cite[Cor.]{To2}
because $\mathfrak  M(H^\infty(\Di))$ satisfies the conditions of the theorem due to the classical work of Su\'arez \cite{S} (see also \cite{Br5} and references therein). Moreover, it is easily deduced from \cite{S} that $\mathfrak  M(H^\infty(V))$ satisfies the conditions of the theorem for $V$ being a Caratheodory hyperbolic Riemann surface of finite type. 

Up until now nothing was known about covering dimension and \v{C}ech cohomology groups of $\mathfrak  M(H^\infty(V))$ for infinite type Riemann surfaces $V$. 
Theorem \ref{te1.1} fills in this gap and provides many examples of domains satisfying conditions (i)--(iii). 

Another class of examples of a different nature of Riemann surfaces satisfying these conditions is given by the following result.
\begin{Th}\label{appen}
Let $C$ be an unbranched covering of a bordered Riemann surface. Then 
\begin{itemize}
\item[(a)] $C$ is dense in $\mathfrak M(H^\infty(C))$;\smallskip
\item[(b)] ${\rm dim}\,\mathfrak M(H^\infty(C))=2$;\smallskip
\item[(c)] $H^2(\mathfrak M(H^\infty(C)),\Z)=0$.
\end{itemize}
\end{Th}
Thus the class of domains $V$ in Theorem \ref{te1.1} includes plain unbranched coverings of bordered Riemann surfaces and domains obtained from them by deleting compact subsets of the analytic capacity zero (e.g., totally disconnected compact subsets), see Example \ref{ex1.3}.
Starting from such a $V$ one can construct a descending chain of $\mathscr B$-domains $V:=V_0\supset V_1\supset V_2\supset\dots\supset V_n$, $n\in\N$, such that each $V_i$ is defined by deleting a hyperbolically-rare sequence of closed disks and then  a compact subset of the analytic capacity zero from $V_{i-1}$. Then all $V_i$ satisfy assumptions of Theorem \ref{te1.1} and, in particular, all algebras $H^\infty(V_i)$ are projective free (and so Hermite).
\begin{E}\label{ex1.3}
{\rm It is well known that every bordered Riemann surface $S$ is a domain in a compact Riemann surface $R$ such that $R\setminus S$ is the disjoint union of finitely many disks  with analytic boundaries. Each $S$ is the quotient of a plane domain $\Omega$ by the discrete action of a Schottky group $G$ (the free group with $g$ generators, where $g$ is the genus of $S$) by M\"{o}bius transformations, see, e.g., \cite{M}. The corresponding quotient map $r: \Omega\rightarrow S$ determines the regular covering of $S$ with the deck transformation group $G$. Then $V:=r^{-1}(R)\subset \Omega$ is a regular covering of $S$ satisfying conditions of Theorem \ref{te1.1}. By definition, $V$ is the complement in $\Omega$ of the finite disjoint union of $G$-orbits of compact simply connected domains with analytic boundaries biholomorphic by $r$ to the connected components of $R\setminus S$. 

Further, if we consider the universal covering $r_u: X\rightarrow S$ of $S$ (where $X=\Di$ if $g\ge 2$, $X=\Co$ if $g=1$ and $X=\Co\mathbb P$ if $g=0$), then  $V:=r_u^{-1}(R)\subset X$ satisfies conditions of Theorem \ref{te1.1} as well. Here $V$ is the complement in $X$ of the finite disjoint union of orbits under the action by M\"{o}bius transformations of the fundamental group $\pi_1(R)$ of $R$ of compact simply connected domains with analytic boundaries biholomorphic by $r_u$ to the connected components of $R\setminus S$. 

Using such $V$ we can define a descending chain of $\mathscr B$-domains satisfying conditions of Theorem \ref{te1.1} $V:=V_0\supset V_1\supset V_2\supset\dots\supset V_n$, $n\in\N$, such that each $V_i$ is constructed by deleting a hyperbolically-rare sequence of closed disks and then  a compact subset of the analytic capacity zero from $V_{i-1}$. 
}
\end{E}

Other examples of Riemann surfaces satisfying conditions of Theorem \ref{te1.1} will be presented in a forthcoming paper.  

In light of these results the following question seems plausible.
\begin{Qu}
Are there plain domains $D$ for which either 
${\rm dim}\,\mathfrak M(H^\infty(D))> 2$ or the \v{C}ech cohomology group $H^2(\mathfrak M(H^\infty(D)),\Z)$ is not trivial?
\end{Qu}

\sect{Maximal Ideal Space of $\mathbf{H^\infty(\mathbb D\times\N)}$} 
The proof of Theorem \ref{te1.1} relies on some properties of
the maximal ideal space of the algebra $H^\infty(\mathbb D\times\N)$. Previously, some  results on the structure of this space were obtained in \cite[Sect.\,4]{Be}. In particular, the corona theorem for $H^\infty(\mathbb D\times\N)$ is valid  and follows from Carleson estimates for solutions of the corona problem for $H^\infty(\Di)$, see, e.g., \cite[Thm.\,4.2]{Be}. The main result used in our proofs provides an additional information on the structure of this object.
\begin{Th}\label{te2.1}
We have
\begin{itemize}
\item[(a)] ${\rm dim}\,\mathfrak M(H^\infty(\Di\times\N))=2$;\smallskip

\item[(b)] $H^2(\mathfrak M(H^\infty(\Di\times\N),\Z)=0$.
\end{itemize}
\end{Th}
Part (a) of the theorem follows from a result of independent interest describing the maximal ideal space $\mathfrak M(H^\infty(\Di\times\N))$ by means of $\mathfrak M(H^\infty(\Di))$.

\begin{Th}\label{teo3.1}
$\mathfrak M(H^\infty(\Di\times\N))$ can be covered by interiors of two compact subsets homeomorphic to a subset of $\mathfrak M(H^\infty(\Di))$. 
 \end{Th}
 \begin{proof}
 Since $\N$ and $\Z$ are both countable, manifolds $\Di\times\N$ and $\Di\times\Z$ are biholomorphic. Hence, $\mathfrak M(H^\infty(\Di\times\N))$ and $\mathfrak M(H^\infty(\Di\times\Z))$ are homeomorphic. It is more convenient for us to work with the latter space.
 
 In what follows, $\Di_r(c):=\{z\in\Co\, :\, |z-c|<r\}$, $r>0$, $c\in\Co$, i.e., $\Di:=\Di_1(0)$. For a subset $S$ of a topological space we denote by  $\bar S$ and $\mathring S$ its closure and interior. Also, we set $H^\infty:=H^\infty(\Di)$.
  
Let 
\begin{equation}\label{equ3.1}
\Omega:=\bigl(\Di\cup\Di_1({\scriptstyle \frac 3 2})\bigr)\setminus\bigl\{\mbox{$ \frac 3 2$}\bigr\}.
\end{equation}
The fundamental group $\pi_1(\Omega)$ of $\Omega$ is isomorphic to $\Z$, i.e., $\pi_1(\Omega)=\{a^n\}_{n\in\Z}$ for some $a\in \pi_1(\Omega)$.
Let $r_u:\Di\rightarrow\Omega$ be the universal covering of $\Omega$. The deck transformation group  $\pi_1(\Omega)$ acts discretely on $\mathbb D$ by M\"{o}bius transformations. Since $r_u\in H^\infty$, it extends to a function $\hat r_u\in C(\mathfrak M(H^\infty))$ such that $\hat r_u(\mathfrak M(H^\infty))=\bar\Omega$.  Let $U:=r_u^{-1}(\Di)\subset\Di$. Since each loop in $\Di$ is contractible in $\Omega$, 
\begin{equation}\label{equ3.1a}
U=\bigsqcup_{g\in\pi_1(\Omega)}g(U')
\end{equation}
for some $U'\subset\Di$ biholomorphic to $\Di$ via $r_u$. In particular, the map  $s: U\rightarrow \Di\times\Z$, 
\begin{equation}\label{equ3.2}
s(z):=(r_u(z),n),\quad  z\in a^n(U'),\quad  n\in\Z,
\end{equation}
is biholomorphic; hence, the pullback by $s$ defines an isomorphism of Banach algebras $s^*: H^\infty(\mathbb D\times\Z)\rightarrow H^\infty(U)$. 

We denote by $\tilde s: \mathfrak M(H^\infty(U))\rightarrow \mathfrak M(H^\infty(\Di\times\Z))$, $\tilde s|_U=s$,  the homeomorphism of the maximal ideal spaces induced by the transpose $(s^*)^*$ of $s^*$.\footnote{Here and below we identify $U$ and $\Di\times\Z$ with their images under the natural embeddings in the corresponding maximal ideal spaces.}

Next, we consider the function $h:={\rm Re}\, r_u$
and its extension $\hat h:={\rm Re}\, \hat r_u\in C(\mathfrak M(H^\infty))$. 
By the definition of $\Omega$, the open set 
\begin{equation}\label{equ3.3}
U_1:=\left\{z\in\Di\, :\, h(z)<\mbox{$\frac 3  4$}\right\}\subset U
\end{equation}
is the preimage under $r_u$ of the set $\{z\in\Di\, :\, {\rm Re}(z)<\frac 3 4\}\subset\Di$. By the corona theorem $U_1$ is dense in 
\begin{equation}\label{equ3.4}
\widetilde U_1:=\left\{\xi\in \mathfrak M(H^\infty)\, :\, \hat h(\xi)<\mbox{$\frac 3  4$}\right\}.
\end{equation}
Moreover, by  \cite[Thm.\,3.2]{S}, each $f\in H^\infty(U_1)$ extends to a (unique) $\hat f\in C(\widetilde U_1)$. In particular, this is valid for  $f\in H^\infty(U)|_{U_1}$.
This determines an isometric homomorphism of Banach algebras $e: H^\infty(U)\rightarrow C(U\cup \widetilde U_1)$ whose transpose $e^*$ induces a continuous injection of $U\cup\widetilde U_1$ into $\mathfrak M(H^\infty(U))$ such that $e^*|_U={\rm id}|_{U}$.
 
Further, 
\begin{equation}\label{equ3.5}
\widetilde V:=\left\{\xi\in \mathfrak M(H^\infty)\, :\, \hat h(\xi)\le\mbox{$\frac 1 2 $}\right\}\subset\widetilde U_1.
\end{equation}
Then $e^*$ maps the compact set $\widetilde V$ homeomorphically onto its image in $\mathfrak M(H^\infty(U))$. 

We have 
\[
(s\circ e^*)(U)=\Di\times\Z\quad {\rm and}\quad p\circ s\circ e^*|_U=r_u|_U,
\]
where $p:\Di\times\Z\rightarrow\Di$ maps $(z,n)$ to $z$, $z\in \Di$, $n\in\Z$.   

Thus,
\[
\hat p\circ \tilde s\circ e^*|_{U\cup \widetilde U_1}=\hat r_u|_{U\cup \widetilde U_1};
\]
here $\hat p: \mathfrak M(H^\infty (\Di\times\Z))\rightarrow\bar\Di$ is the extension of $p$ via the Gelfand transform.
In particular, if $S:=\{z\in\bar \Di\, :\, {\rm Re}(z)\le\frac 1 2\}$, then $\tilde s\circ e^*$  maps $\widetilde V$ homeomorphically onto $\hat p^{-1}(S)$. We set $S_o:=\{z\in\bar\Di\, :\, -z\in S\}:=\{z\in\bar \Di\, :\, {\rm Re}(z)\ge -\frac 1 2\}$. Then $\bar \Di=\mathring S\cup \mathring S_o$ so that 
\[
\mathfrak M(H^\infty (\Di\times\Z))=\hat p^{-1}(\bar\Di)=\hat p^{-1}(\mathring S)\cup \hat p^{-1}(\mathring S_o)=\hat p^{-1}(S)\cup \hat p^{-1}(S_o),
\] 
where each term in the last expression is homeomorphic to $\widetilde V$.
\end{proof}

\begin{proof}[Proof of Theorem \ref{te2.1}]
(a) Since $\mathfrak M(H^\infty (\Di\times\Z))=\hat p^{-1}(S)\cup \hat p^{-1}(S_o)$, where each of the subsets is homeomorphic to the compact subset $\widetilde V\subset\mathfrak M(H^\infty)$, see above, and  ${\rm dim}\, \mathfrak M(H^\infty)=2$,  ${\rm dim}\,\mathfrak M(H^\infty (\Di\times\Z))\le {\rm dim}\,\widetilde V\le 2$. It equals 2 as $\mathfrak M(H^\infty (\Di\times\Z))$ contains the 2-dimensional subset $\Di\times\Z$.\medskip

\noindent (b) Since ${\rm dim}\,\mathfrak M(H^\infty(\Di\times\N))=2$, by the Hopf theorem, see, e.g., \cite{Hu}, elements of the cohomology group $H^2(\mathfrak M(H^\infty(\Di\times\N)),\Z)$ are in a one-to-one correspondence with elements of the set $[\mathfrak M(H^\infty(\Di\times\N)),\mathbb S^2]$ of homotopy classes of continuous maps of $\mathfrak M(H^\infty(\Di\times\N))$ to the two-dimensional unit sphere $\mathbb S^2$. In turn, according to the Novodvorskii-Taylor theory, there is a bijection of $[\mathfrak M(H^\infty(\Di\times\N)),\mathbb S^2]$ onto the set $[{\rm ID}_1(H^\infty(\Di\times\N)_2)]$ of connectivity components of the class of idempotents consisting of $2\times 2$ matrices with entries in $H^\infty(\Di\times\N)$ of constant rank $1$, see
\cite[Sec.\,5.3,\,page\,186]{Ta}. Thus to show that $H^2(\mathfrak M(H^\infty(\Di\times\N)),\Z)=0$ we must prove that each idempotent $F\in {\rm ID}_1(H^\infty(\Di\times\N)_2)$ is similar (by an invertible $2\times 2$ matrix with entries in $H^\infty(\Di\times\N)$) to a matrix of the form
\begin{equation}\label{eq7.1}
{\rm diag}(1,0):=\left[
\begin{array}{cc}
1&0\\
0&0
\end{array}
\right].
\end{equation}
The proof follows the lines of the proof of the Theorem in \cite[Sect.\,5]{Br5}.

We set 
\begin{equation}\label{eq7.2}
F_m:=F|_{\Di\times\{m\}},\quad m\in\N.
\end{equation}
Then $N_{1m}=\ker(F_m)$ and $N_{2m}=\ker(I_2-F_m)$, where  $I_2$ is the identity $2\times 2$ matrix, are weak$^*$ closed $H^\infty$-submodules of $H_2^\infty\, (:= H^\infty\oplus H^\infty)$ (i.e., if $\{f_n\}_{n\in\N}\subset N_{im}$ is a bounded sequence pointwise converging to $f\in H_2^\infty$, then $f\in N_{im}$). Since $F_m$ is of constant rank $1$, the famous Beurling-Lax-Halmos theorem, see, e.g., \cite{Ni}, \cite[p.\,1025]{To1}, implies that $N_{im}=H_{im}\cdot H^\infty$, where $H_{im}$ is a $2\times 1$ matrix with entries in $H^\infty$ of constant rank $1$ such that
\begin{equation}\label{eq7.3}
H_{im}^* \cdot H_{im}=1\quad {\rm a.e.\ on}\ \mathbb S:=\{z\in\Co\, :\, |z|=1\}.
\end{equation}
Hence, since the columns of $I_2-F_m$  and of $F_m$ belong to $N_{1m}$ and $N_{2m}$, respectively,
$I_2-F_m=H_{1m}\cdot G_{1m}$ and $F_m=H_{2m}\cdot G_{2m}$ for some $1\times 2$ matrices $G_{im}$ with entries in $H^\infty$ of constant rank $1$. Then we have
\begin{equation}\label{eq7.4}
\begin{array}{l}
\displaystyle H_{1m}\cdot G_{1m}=(I_2-F_m)=(I_2-F_m)^2=H_{1m}\cdot (G_{1m}\cdot H_{1m})\cdot G_{1m},\\
\\
H_{2m}\cdot G_{2m}=F_m=F_m^2=H_{2m}\cdot (G_{2m}\cdot H_{2m})\cdot G_{2m},\\
\\
(H_{1m}\cdot G_{1m})(H_{2m}\cdot G_{2m})=(H_{2m}\cdot G_{2m})(H_{1m}\cdot G_{1m})=0.
\end{array}
\end{equation}
These,  \eqref{eq7.3} and the maximum modulus principle for $H^\infty$ imply that
\begin{equation}\label{eq7.5}
G_{1m}\cdot H_{1m}=1,\quad G_{2m}\cdot H_{2m}=1,\quad G_{1m}\cdot H_{2m}=G_{2m}\cdot H_{1m}=0.
\end{equation}
Let us define $2\times 2$ matrices with entries in $H^\infty$
\[
H_m:=(H_{1m}\, H_{2m})\quad {\rm and}\quad G_m:=\left(\!\!\!
\begin{array}{c}
G_{1m}\\
G_{2m}
\end{array}
\!\!\!\right).
\]
Then  \eqref{eq7.5}, the definitions of $G_{im}$ and \eqref{eq7.3} imply that
\begin{equation}\label{eq7.6}
\begin{array}{l}
\displaystyle
G_m=H_m^{-1},\qquad H_{m}^{-1}\cdot F_m\cdot H_m={\rm diag}(1,0)\quad {\rm and}\\
\\
\displaystyle \|H_m\|_{{\scriptstyle H_{2}^\infty\rightarrow H_2^\infty}}\le 1,\qquad \|H_m^{-1}\|_{{\scriptstyle H_{2}^\infty\rightarrow H_2^\infty}}\le 1+\|F_m\|_{{\scriptstyle H_{2}^\infty\rightarrow H_2^\infty}}.
\end{array}
\end{equation}
Finally, let us define a $2\times 2$ matrix $H$ on $\Di\times\N$ by the formula
\begin{equation}\label{eq7.7}
H|_{\Di\times\{m\}}=H_m,\quad m\in\N.
\end{equation}
Since $\sup_{m\in\N}\|F_m\|_{{\scriptstyle H_{2}^\infty\rightarrow H_2^\infty}}<\infty$, \eqref{eq7.6} shows that $H$ is an invertible $2\times 2$ matrix with entires in $H^\infty(\Di\times\N)$ such that
$H^{-1}|_{\Di\times\{m\}}=H_m^{-1}$, $m\in\N$, and $H^{-1}\cdot F\cdot H={\rm diag}(1,0)$.
 
This completes the proof of part (b) of the theorem.
\end{proof}

\sect{Proofs of Theorem \ref{te1.1} and Corollary \ref{cor1.2}}
\subsect{Structure of $\mathfrak M(H^\infty(U))$}
Let  $U$ be a $\mathscr B$-domain obtained from a domain $V\subset\Co$ by
deleting a (possibly finite) hyperbolically-rare sequence of closed disks $\{\Delta_n\}\subset V$ with centers $\alpha_n$ and let $V$ be dense in $\mathfrak M(H^\infty(V))$. Then by \cite[Thm.\,6.1]{Be}
$U$ is dense in $\mathfrak M(H^\infty(U))$ and the latter space has the described below structure.

Let $R: \mathfrak M(H^\infty(U))\rightarrow \mathfrak M(H^\infty(V))$ be the continuous map transposed to the restriction homomorphism $H^\infty(V)\rightarrow H^\infty(U)$, $f\mapsto f|_U$. If the sequence $\{\alpha_n\}$ is infinite, we denote by $F:\beta\N\rightarrow\mathfrak M(H^\infty(V))$ the continuous extension of the map $\N\rightarrow V$, $n\mapsto\alpha_n$, to the Stone-\v{C}ech compactification of $\N$ and set 
\begin{equation}\label{Set}
S:=F(\beta\N\setminus\N)\subset \mathfrak M(H^\infty(V))\setminus V.
\end{equation}
Under these notations we have following properties.
\begin{itemize}
\item[(1)]  There is a continuous embedding $i:\bigl(\mathfrak M(H^\infty(V))\setminus V\bigr)\cup U\hookrightarrow\mathfrak M(H^\infty(U))$ such that $R\circ i={\rm id}$ and $i$ is invertible on $\bigl(\mathfrak M(H^\infty(V))\setminus (V\cup S)\bigr)\cup U$, \cite[Thm.\,3.1]{Be}.\smallskip
\item[(2)]  There is a continuous mapping $G:\mathfrak M(H^\infty(\Di\times\N))\setminus (\Di\times\N)\rightarrow R^{-1}(S)$ which maps $\mathfrak M(H^\infty(\Di\times\N))\setminus \bigl((\Di\times\N)\cup (\{0\}\times (\beta\N\setminus\N))\bigr)$ homeomorphically onto $R^{-1}(S)\setminus i(S)$ and $\{0\}\times (\beta\N\setminus\N)$ surjectively onto $i(S)$,
\cite[Thm.\,6.2]{Be}.\smallskip
\item[(3)] $R^{-1}(\partial\Delta_n)$  is homeomorphic to $\mathfrak M(H^\infty(\Di))\setminus\Di$ for all $n\in\N$.
 \end{itemize}
 The last property easily follows from the decomposition of $H^\infty(U)$, see \cite[Sec.\,2]{Be}.
 
In turn, if the sequence $\{\alpha_n\}$ is finite, then $\mathfrak M(H^\infty(U))$ satisfies property (1) with $S=\emptyset$ and property (2).
 \subsect{Proof of Theorem \ref{te1.1}}
 We prove the theorem under the condition that $\{\alpha_n\}$ is infinite. If it is finite, then the proof is simpler but still repeats some of the arguments presented below. We leave the details to the reader.
\begin{proof}
First, we prove that under conditions of the theorem ${\rm dim}\,\mathfrak M(H^\infty(U))=2$.
 
 To this end, it suffices to prove that the compact set $K_1:=R^{-1}((V\setminus U)\cup S)$ has dimension $\le 2$ and each compact subset of the open set $K_2:=\mathfrak M(H^\infty(U))\setminus K_1$ has dimension $\le 2$ as well, see, e.g., \cite[Ch.\,2,\,Thm.\,9-11]{N}. 
 
 By this definition, 
 \[
 K_2=R^{-1}\bigl(U\cup \bigl(\mathfrak M(H^\infty(V))\setminus (V\cup S)\bigr)\bigr).
 \]
 Hence, since $R|_{U}={\rm id}$, due to property (1) each compact subset $Z\subset K_2$ is homeomorphic to a compact subset of $U\cup (\mathfrak M(H^\infty(V))\setminus (V\cup S))\subset \mathfrak M(H^\infty(V))$. In particular, ${\rm dim}\, Z\le 2$  because ${\rm dim}\,\mathfrak M(H^\infty(V))=2$ by the hypotheses of the theorem.
 
Next, 
\begin{equation}\label{ka1}
K_1=\bigsqcup_n R^{-1}(\partial\Delta_n)\sqcup R^{-1}(S);
\end{equation}
here $\partial \Delta_n$ stands for the boundary of $\Delta_n$.

Hence, due to \cite[Ch.\,2,\,Thm.\,9-11]{N} ${\rm dim}\, K_1\le 2$ iff
${\rm dim}\, R^{-1}(S)\le 2$ and ${\rm dim}\, Z\le 2$ for each compact subset $Z\subset\sqcup_n\, R^{-1}(\partial\Delta_i)$. 

Since each $R^{-1}(\partial\Delta_i)$ is a relatively clopen subset of $K_1$, there is some $n=n(Z)$ such that $Z\subset \sqcup_{1\le k\le n}\, R^{-1}(\partial\Delta_k)$. Then according to property (3) and \cite[Thm.\,4.5]{S}
\[
{\rm dim}\, Z\le\max_{1\le k\le n}{\rm dim}\,R^{-1}(\partial\Delta_k)={\rm dim}\,\bigl(\mathfrak M(H^\infty)\setminus\Di\bigr)=2.
\]

Thus it remains to show that ${\rm dim}\, R^{-1}(S)\le 2$. We have
\[
R^{-1}(S)=(R^{-1}(S)\setminus i(S))\sqcup i(S).
\]
According to property (2) and Theorem \ref{te2.1} for each compact subset $Z\subset R^{-1}(S)\setminus i(S)$ we obtain
\[
{\rm dim}\,Z\le {\rm dim}\,\left(\mathfrak M(H^\infty(\Di\times\N))\setminus \bigl((\Di\times\N)\cup (\{0\}\times (\beta\N\setminus\N))\bigr)\right)\le {\rm dim}\,\mathfrak M(H^\infty(\Di\times\N))=2.
\]
Also, according to property (1) and the hypotheses of the theorem
\[
{\rm dim}\,i(S)={\rm dim}\, S\le 2.
\]
This and \cite[Ch.\,2,\,Thm.\,9-11]{N} imply that ${\rm dim}\, R^{-1}(S)\le 2$. Combining all preceding inequalities and using that ${\rm dim}\, U=2$ we obtain that ${\rm dim}\,\mathfrak M(H^\infty(U))=2$, as required.\smallskip

Now, let us prove that under the hypotheses of the theorem $H^2(\mathfrak M(H^\infty(U)),\Z)=0$. 

To this end, using the natural one-to-one correspondence
between isomorphism classes of rank 1 complex vector bundles on $\mathfrak M(H^\infty(U))$ and elements of $H^2(\mathfrak M(H^\infty(U)),\Z)$ given by the first Chern classes of the bundles, it suffices to prove that each rank $1$ complex vector bundle $E$ on $\mathfrak M(H^\infty(U))$ is trivial, i.e., admits a nowhere vanishing continuous section. We prove this in several stages.

First, let us show that $E|_{i(S)}$ is trivial. 

Indeed, since  by the hypotheses ${\rm dim}\,\mathfrak M(H^\infty(V))=2$ and $H^2(\mathfrak M(H^\infty(V),\Z)=0$, by property (1) employing the long cohomology sequence of the pair $\bigl(\mathfrak M(H^\infty(V)),i(S)\bigr)$ we obtain that $H^2(i(S),\Z)=0$. Hence, the first Chern class of $E|_{i(S)}$ is 0, i.e.,
the bundle $E|_{i(S)}$ is trivial.  By $s_1$ we denote a nowhere continuous section of $E|_{i(S)}$.

Next, we show that  $E|_{R^{-1}(S)}$  is trivial.

Let $G^*(E|_{R^{-1}(S)})$ be the pullback by $G$ of the bundle $E|_{R^{-1}(S)}$ to the compact set $\mathfrak M(H^\infty(\Di\times\Z))\setminus (\Di\times\N)$. Since $\mathfrak M(H^\infty(\Di\times\Z))=2$ and $H^2(\mathfrak M(H^\infty(\Di\times\Z)),\Z)=0$ by Theorem \ref{te2.1},
as above we obtain that $H^2\bigl(\mathfrak M(H^\infty(\Di\times\Z))\setminus (\Di\times\N),\Z\bigr)=0$. Hence,  the bundle $G^*(E|_{R^{-1}(S)})$ is trivial, i.e. has a nowhere continuous section, say, $s_2$.

Due to property (2), $G^*s_1$ is a nowhere continuous section of  the restriction of the bundle $G^*(E|_{R^{-1}(S)})$  to $G^{-1}(i(S))=\{0\}\times (\beta\N\setminus\N)$. Hence, 
$f:=(G^*s_1)\cdot s_2^{-1}|_{\{0\}\times (\beta\N\setminus\N)}$ is a nowhere vanishing continuous function on $\{0\}\times (\beta\N\setminus\N)$.

Further, the pullback by means of the projection $\Di\times\N\rightarrow\N$, $(z,n)\mapsto n$, determines a monomorphism of Banach algebras $\ell^\infty\hookrightarrow H^\infty(\Di\times\N)$ whose transpose induces a continuous surjection $P:\mathfrak M(H^\infty(\Di\times\N))\rightarrow\beta(\N)$ such that
$P^{-1}(\N)=\Di\times\N$. In particular, $P$ maps the compact set $\mathfrak M(H^\infty(\Di\times\N))\setminus (\Di\times\N)$ onto $\{0\}\times (\beta\N\setminus\N)$.

We set
\[
g:=P^*f\in C\bigl(M(H^\infty(\Di\times\N))\setminus (\Di\times\N),\Co^*\bigr).
\]
Then $g\cdot s_2$ is nowhere vanishing continuous section of $G^*(E|_{R^{-1}(S)})$ whose restriction to $G^{-1}(i(S))$ coincides with $G^*s_1$. In turn, according to property (2), there is a continuous nowhere vanishing section $s_3$ of $E|_{R^{-1}(S)}$ such that $s_3|_{i(S)}=s_1$ and $G^*s_3=g\cdot s_2$. This shows that the bundle
$E|_{R^{-1}(S)}$  is trivial.

Now, let us show that $E|_{K_1}$ is trivial, see \eqref{ka1}. 

Indeed, since $E|_{R^{-1}(S)}$  is trivial, there is a relatively open neighbourhood $O\subset K_1$ of $R^{-1}(S)$ such that $E|_O$ is trivial (this follows from the standard extension property of sections of vector bundles). Then the compact set
$K_1\setminus O$ is covered by relatively clopen pairwise disjoint sets $R^{-1}(\partial\Delta_n)$, $n\in\N$. In particular, there is some $n_0\in\N$ such that
\begin{equation}\label{embed}
K_1\setminus O\subset \bigsqcup_{1\le k\le n_0}R^{-1}(\partial\Delta_k).
\end{equation}
According to property (3) and \cite[Thm.\,4.5]{S} 
\[
{\rm dim}\,\left(\bigsqcup_{1\le k\le n_0}R^{-1}(\partial\Delta_k)\right)=2\quad {\rm and}\quad H^2\left(\bigsqcup_{1\le k\le n_0}R^{-1}(\partial\Delta_k),\Z\right)=0.
\]
Hence, the restriction of $E$ to $\sqcup_{1\le k\le n_0}R^{-1}(\partial\Delta_k)$ is trivial. 

Then $\bigl(\{R^{-1}(\partial\Delta_k)\}_{1\le k\le n_0}, K_1\setminus\bigl(\sqcup_{1\le k\le n_0}R^{-1}(\partial\Delta_k)\bigr)\bigr)$ is an open cover of $K_1$ by pairwise disjoint relatively open sets and the restriction of $E$ to each of the sets is trivial, see \eqref{embed}. Thus, the bundle $E|_{K_1}$ is trivial.

Finally, we prove that the bundle $E$ is trivial.

Due to property (1) and the hypotheses of the theorem the restriction of $E$ to each compact subset of $i\bigl((\mathfrak M(H^\infty(V))\setminus V)\cup U\bigr)$ is trivial. On the other hand, due to the previous statement there is an open neighbourhood $O_1$ of $K_1$ such that $E|_{O_1}$ is trivial. Let $O_2$ be another open neighbourhood of $K_1$ such that $\overline{O}_2\subset O_1$. Then $\mathfrak M(H^\infty(U))\setminus O_2$ is a compact subset of $i\bigl((\mathfrak M(H^\infty(V))\setminus V)\cup U\bigr)$ and $\bigl(\mathfrak M(H^\infty(U))\setminus \overline{O}_2, O_1\bigr)$ is an open cover of $\mathfrak M(H^\infty(U))$. By the definition, 
\[
Y_1:=i^{-1}\bigl(\mathfrak M(H^\infty(U))\setminus \overline{O}_2\bigr)
\quad {\rm 
 and}\quad  Y_2:=i^{-1}\bigl(O_1\setminus K_1\bigr)
  \]
  are open subsets of $(\mathfrak M(H^\infty(V))\setminus V)\cup U$ and  $\bigl((\mathfrak M(H^\infty(V))\setminus V)\cup U\bigr)\setminus S$, see \eqref{Set}, such that $Y_1\cup Y_2$ covers $(\mathfrak M(H^\infty(V))\setminus V)\cup U$. Moreover,
\begin{equation}\label{exten}
Y_3:=Y_2\cup S\cup\left(\bigsqcup_i\,\overline{\Delta}_i\right)
\end{equation}
is an open neighbourhood of the compact set $S\cup\bigl(\sqcup_{i}\,\overline{\Delta}_i\bigr)\subset \mathfrak M(H^\infty(V))$ such that 
\[
R^{-1}(Y_3)=O_1.
\]

Let $t_1$ and $t_2$ be continuous nowhere vanishing sections of the restrictions of $E$ to $O_1$ and $\mathfrak M(H^\infty(U))\setminus \overline{O}_2$, respectively (existing by the previous arguments). Then
\[
t_{12}:=t_1^{-1}\cdot t_2\in C\bigl((\mathfrak M(H^\infty(U))\setminus \overline{O}_2)\cap O_1,\Co^*\bigr).
\]
By our construction, $(\mathfrak M(H^\infty(U))\setminus \overline{O}_2)\cap O_1\subset i\bigl((\mathfrak M(H^\infty(V))\setminus V)\cup U\bigr)$ and 
\[
i^{-1}\bigl((\mathfrak M(H^\infty(U))\setminus \overline{O}_2)\cap O_1\bigr) =Y_1\cap Y_3.
\]
Hence, the pullback $(i^{-1})^* t_{12}$ by $i^{-1}$ is a $1$-cocycle on the cover $(Y_1,Y_3)$ of $\mathfrak M(H^\infty(V))$ with values in $\Co^*$. Each such cocycle determines a complex rank 1 vector bundle on $\mathfrak M(H^\infty(V))$ which according to the hypotheses of the theorem is trivial. Hence, there exist $\tilde t_1\in C(Y_3,\Co^*)$ and $\tilde t_2\in C(Y_1,\Co^*)$ such that
\[
\tilde t_2^{-1}\cdot\tilde t_1=(i^{-1})^* t_{12}\quad {\rm on}\quad Y_1\cap Y_3.
\]
The latter implies that
\[
(R^*\tilde t_2)^{-1}\cdot (R^*\tilde t_1)=R^*((i^{-1})^*t_{12})=t_1^{-1}\cdot t_2\quad {\rm on}\quad (\mathfrak M(H^\infty(U))\setminus \overline{O}_2)\cap O_1,
\]
because $R^*(i^{-1})^*:=(i^{-1}\circ R)^*={\rm id}$ on $\bigl((\mathfrak M(H^\infty(V))\setminus V)\cup U\bigr)\setminus S$ by property (1).
Here $R^*\tilde t_1\in C(O_1,\Co^*)$ and $R^* \tilde t_2\in C\bigl(\mathfrak M(H^\infty(U))\setminus \overline{O}_2 ,\Co^*\bigr)$. 

Thus continuous nowhere vanishing sections $t_1\cdot R^*\tilde t_1$ of $E|_{O_1}$ and $t_2\cdot R^*\tilde t_2$ of $E|_{\mathfrak M(H^\infty(U))\setminus \overline{O}_2}$ coincide on $(\mathfrak M(H^\infty(U))\setminus \overline{O}_2)\cap O_1$ and so determine a continuous nowhere vanishing section of $E$, i.e., the bundle $E$ is trivial.

The proof of the theorem is complete.
\end{proof}
\subsect{Proof of Corollary \ref{cor1.2}} The result follows from Theorem \ref{te1.1} by \cite[Cor.\,1.4]{BS}.

\sect{Proof of Theorem \ref{appen}}
Let $r:S'\rightarrow S$ be an unbranched covering of a bordered Riemann surface $S$. We have to prove that
\begin{itemize}
\item[(a)] $S'$ is dense in $\mathfrak M(H^\infty(S'))$;\smallskip
\item[(b)] ${\rm dim}\,\mathfrak M(H^\infty(S'))=2$;\smallskip
\item[(c)] $H^2(\mathfrak M(H^\infty(S')),\Z)=0$.
\end{itemize}

In fact, part (a) was proved in \cite[Cor.\,1.6]{Br1} and part (c) follows from the projective freeness of $H^\infty(S')$ established in \cite[Thm.\,1.5]{BS} (see also \cite[Sect.\,3]{Br5}).  Thus it remains to prove part (b) only.  The proof is based on some results and constructions of the theory developed in \cite{Br1}, \cite{Br2}. We refer to these papers for additional details.

\subsect{Auxiliary Results} For  the facts presented in this section see, e.g., \cite[Sec.\,2.2]{Br2},  \cite{Br3} and references therein.
 
 It is well known that $S$ can be regarded as a domain in a compact Riemann surface $R$ such that
\begin{equation}\label{eq4.1}
R\setminus \bar S=\bigsqcup_{i=1}^k D_i,
\end{equation}
where $D_i$ are open disks with analytic boundaries.

Let  $S_o\subset R$ be another bordered Riemann surface containing $\bar S$ as a deformation retract. By the covering homotopy theorem there is an unbranched covering $r: S_{o}'\rightarrow S_o$ such that $S'\subset S_{o}'$ is a domain and $r|_{S'}: S'\rightarrow S$ is the given unbranched covering of $S$. 

The covering $r: S_{o}'\rightarrow S_o$  can be viewed as a fiber bundle over $S_o$ with a discrete fiber $F$. Let $E(S_o,\beta F)$ be the space obtained from $S_o'$ by taking the Stone-\v{C}ech compactifications of fibres under $r$. Then $E(S_o,\beta F)$ is a normal Hausdorff space of covering dimension $2$ and $r$ extends to a continuous map $r_E: E(S_o,\beta F)\rightarrow S_o$ such that $\bigl(E(S_o,\beta F),S_o,r_E,\beta F\bigr)$ is a fibre bundle over $S_o$ with fibre $\beta F$ and $S_o'$ embeds in $E(S_o,\beta F)$ as an open dense subbundle. 

If $K\subset S_o$ is a compact set and $K':=r^{-1}(K)$,  $K_E:=r_E^{-1}(K)$, then $K'$ is dense in $K_E$ and a bounded continuous function on $K'$ admits a continuous extension to the compact set $K_E$ if and only if it is uniformly continuous with respect to the path metric induced by a Riemannian metric pulled back by $r$ from $S_o$.  In particular, this is valid for restrictions to $K'$ of bounded holomorphic functions defined on the preimage by $r$ of a neighbourhood of $K$. This implies that
each function in $H^\infty(S')$ extends continuously to $S_E$. 
Moreover, the algebra of such extensions separates the points of $S_E$ so that there is an injective continuous map  of $S_E$  into $\mathfrak M(H^\infty(S'))$.
In what follows, we identify $S_E$ with its image under the embedding. Then $S_E$ is a dense subset of
$\mathfrak M(H^\infty(S'))$. Similarly, we regard $E(S_o,\beta F)$ as a dense subset of $\mathfrak M(H^\infty(S_o'))$.
\subsect{Proof of Theorem \ref{appen}}
We retain notations of the previous section.

Let  $\hat r:\mathfrak M(H^\infty(S'))\rightarrow\bar S$ be the continuous surjective map induced by 
the transpose of the homomorphism
$H^\infty(S_o)\rightarrow H^\infty(S')$, $f\mapsto f\circ r|_{S'}$. By definition,  $\hat r|_{S'}=r$.
\begin{Lm}\label{lem4.1}
The set $\mathfrak M(H^\infty(S'))\setminus \hat r^{-1}(\partial S)$ coincides with $S_E\, (:=r^{-1}(S))$. 
\end{Lm}
\begin{proof}
Let $x\in \mathfrak M(H^\infty(S'))\setminus \hat r^{-1}(\partial S)$ and $\{x_\alpha\}\subset S'$ be a net converging to $x$. Since $\hat r(x)\in S$ and $\hat r(x)=\lim_\alpha r(x_\alpha)$, passing to a subnet, if necessary, without loss of generality we may assume that all points $r(x_\alpha)$ belong to a neighbourhood $U$ of $x$ such that $\bar U\subset S$. Then by the definition of $E(S_o,\beta F)$, $r_E^{-1}(\bar U)$ is a compact subset of $S_E$ in the original topology of $E(S_o,\beta F)$ containing all points $x_\alpha$.
Since the image of a compact set under a continuous map is compact, $r_E^{-1}(\bar U)$ as a subset of  $\mathfrak M(H^\infty(S'))$ is compact in the Gelfand topology. This implies that $x=\lim_\alpha x_\alpha\in r_E^{-1}(\bar U)\subset 
S_E$, as required.
\end{proof}
The lemma implies that $S_E$ is an open dense subspace of $\mathfrak M(H^\infty(S'))$ and 
\[
\mathfrak M(H^\infty(S'))=S_E\sqcup \hat r^{-1}(\partial S).
\]
Then, since ${\rm dim}\, S_E=2$ and for each compact subset $K\subset S_E$ the restriction of the Gelfand topology of $\mathfrak M(H^\infty(S'))$ to $K$ coincides with the topology induced from $E(S_o,\beta F)$, ${\rm dim}\, K\le 2$, where dimension is defined by open covers in the Gelfand topology. Thus, to show that ${\rm dim}\,\mathfrak M(H^\infty(S'))=2$ it suffices to show that ${\rm dim}\, \hat r^{-1}(\partial S)\le 2$, see \cite[Ch.2,\,Thm.9-11]{N}.
To this end, we prove that $\hat r^{-1}(\partial S)$ can be covered by finitely many compact subsets homeomorphic to subsets of $\mathfrak M(H^\infty(\Di\times\N))$. This and Theorem \ref{te2.1}(a) will imply the required statement.\smallskip

Let $p:\mathfrak M(H^\infty(S'))\rightarrow\mathfrak M(H^\infty(S_{o}'))$
be the continuous map induced by the
transpose of the homomorphism $H^\infty(S_o')\rightarrow H^\infty(S')$, $f\mapsto f|_{S'}$. Then  $p|_{S_E}={\rm id}$, 
the image of $p$ is $\bar S_E:=r_E^{-1}(\bar S)\, (\subset E(S_o,\beta F))$ and
\begin{equation}\label{e4.2}
\hat r=r_E\circ p.
\end{equation}
Since $\bar S_E$ is a compact subset of $E(S_o,\beta F)$ in the original topology of $E(S_o,\beta F)$, it is a compact subset of $\mathfrak M(H^\infty(S_{o}'))$ as well.
Therefore 

\noindent (*) {\em if $f$ is a bounded continuous function on $S'$ which extends to a function $\tilde f$ on $\bar S_E$ continuous in the bundle topology of $E(S_o,\beta F)$, then  $p^*\tilde f$ is an extension of $f\, (=p^*f)$ to $\mathfrak M(H^\infty(S'))$ continuous in the Gelfand topology.}\smallskip

 \noindent This fact will be used  in the proof.

Let us proceed with the proof of the statement ${\rm dim}\, \hat r^{-1}(\partial S)=2$.
To this end, we choose some open disks $\widetilde D_i$ containing $D_i$ such that each
$A_i:=\widetilde D_i\setminus D_i\subset S$ is biholomorphic up to the boundary to an annulus $A:=\{z\in\Co\, :\, c<|z|<1\}$ with $\partial D_i$ homeomorphic to the outer boundary circle, see \eqref{eq4.1}. We set
\[
A_i':=r^{-1}(A_i),\quad 1\le i\le k.
\]
By definition, each connected component of $A_i'$ is biholomorphic either to an annulus or $\Di$. (If the covering is not regular, then these components are not necessarily biholomorphic.) We cover $A_i$ by two open sets $A_{i1}$, $A_{i2}$ biholomorphic (by the above chosen biholomorphism of $A_i$ and $A$) to 
\[
A_{1}=\left\{z:=re^{i\theta}\in A\, :\, -\pi<\theta<\frac{\pi}{2}\right\}\quad {\rm and}\quad  A_2=\left\{z:=re^{i\theta}\in A\, :\, 0<\theta<\frac{3\pi}{2}\right\},
\] 
respectively, and set 
\[
A_{ij}':=r^{-1}(A_{ij}).
\]
By definition, each connected component of $A_{ij}'$ is biholomorphic by means of $r$ to $A_{ij}$; thus, $A_{ij}'$ is biholomorphic to $A_{ij}\times F$.
By $\partial_o A_{ij}:=\overline{A}_{ij}\cap\partial D_i$ we denote the part of the `outer' boundary of $A_{ij}$.
\begin{Prop}\label{prop4.2}
Every holomorphic function from $H^\infty(A_{ij}')$ can be continuously extended to $\hat r^{-1}(A_{ij}\cup \partial_o A_{ij})$.
\end{Prop}
Note that $A_{ij}'$ is dense in $\hat r^{-1}(A_{ij}\cup  \partial_o A_{ij})$ by the corona theorem for $H^\infty(S')$.
\begin{proof}
First, we prove the following result.
\begin{Lm}\label{lem4.3}
Each $f\in H^\infty(A_{i}')$ admits an extension $\tilde f\in C(\hat r^{-1}(A_{i}\cup \partial D_i))$.
\end{Lm}
\begin{proof}
To avoid technicalities we use the convenient language of \cite[Sect.\,3.3]{Br3}.

Let $X:=S_o\times S_{o}'$. We embed $S_{o}'$ into $X$ by the formula
\begin{equation}\label{eq4.2}
e(z):=(r(z), z),\quad z\in S_{o}'.
\end{equation}
Then $e(S_{o}')$ is a closed submanifold of the two-dimensional Stein manifold $X$. 
As follows from \cite[Thm.\,1.3]{Br3} for each  $f\in H^\infty(A_{i}')$ there is some $F\in H^\infty(A_{i}\times S_{o}')$ such that 
\[
e^*F=f.
\]
Such $F$ can be regarded as a $H^\infty$ function on $A_{i}$ with values in the Banach space $H^\infty(S_{o}')$. 

Let $\rho$ be a $C^\infty$ function on $\bar S$ equal to $1$ in a neighbourhood of $\partial D_i$ with support in $A_i$. Then $\overline\partial(\rho F)$ is a
$H^\infty(S_o')$-valued $C^\infty$ $(0,1)$-form on $\bar S$ equals 0 in a neighbourhood of $\partial D_i$. Hence, it is extended by $0$ to a $(0,1)$-form $\omega$ on $S_o$.
According to the generalized H.\,Cartan theory, see \cite{Bu}, equation
\[
\overline\partial G=\omega
\]
is solvable on $S_o$. Its solution $G$ is a $C^\infty$ function on $S_o$ with values in $H^\infty(S_o')$ holomorphic (and bounded) in a neighbourhood of $\partial D_i$ such that
\[
G_1:=\rho F- G|_S\in H^\infty(S,H^\infty(S_o'))=H^\infty(S\times S_o').
\]
Thus, $F=G_1+G$ in a neighbourhood $N$ of $\partial D_i$ in $S$. Considering $G_1$ and $G$ as functions on $S_o\times S_o'$ and restricting them to $e(A_i')$ we obtain 
\begin{equation}\label{e4.4}
f=g_1+g\quad {\rm on}\quad N':=r^{-1}(N);
\end{equation}
here $g_1:=e^*G_1\in H^\infty(S')$ and $g:=e^*G$ is such that  $g|_{N'}\in H^\infty(N')$.  

By definition, $g_1$ admits a continuous extension to $\mathfrak M(H^\infty(S'))$ by the Gelfand transform. Also, $g$ being bounded uniformly continuous on $r^{-1}(\bar S)$ with respect to the path metric induced by a
Riemannian metric pulled back by $r$ from $S_o$, see \cite{Br3}, admits a continuous extension to $\bar S_E\subset\mathfrak M(H^\infty(S_o'))$ and,  hence, to $\mathfrak M(H^\infty(S'))$, see (*). This and \eqref{e4.4} imply that $f|_{N'}$ can be extended to a function  $\tilde f_1\in C(\hat r^{-1}(N\cup\partial D_i))$. On the other hand, $f$ admits a continuous extension to $S_E\subset \mathfrak M(H^\infty(S'))$ denoted by $\tilde f_2$, see Section~4.1.

We have 
\[
\tilde f_1|_{N'}-\tilde f_2|_{N'}=0.
\]
Since the open set $N'$ is dense in $\hat r^{-1}(N)$ by the corona theorem, the previous equation implies that $\tilde f_1=\tilde f_2$ on $\hat r^{-1}(N)$, i.e., these functions coincide on $\hat r^{-1}(N\cup\partial D_i)\cap S_E$ and so define a function $\tilde f\in C(\hat r^{-1}(S\cup\partial D_i))$ which extends $f$.

The proof of the lemma is complete.
\end{proof}

Now, let us complete the proof of the proposition.

As before, for $f\in H^\infty(A_{ij}')$ there is some $F\in H^\infty(A_{ij}\times S_{o}')$ regarded as a $H^\infty$ function on $A_{ij}$ with values in the Banach space $H^\infty(S_{o}')$ such that 
$e^*F=f$, see \cite[Thm.\,1.3]{Br3}. Let $\rho$ be a  $C^\infty$ function on $\bar A_i$ with support in $\overline{A}_{ij}$ equal to 1 on a subset $Z_\rho\subset \overline{A}_{ij}$ which is the intersection of a closed sector with the origin at $0$ with $\overline{A}_{ij}$ (here we identify $\bar A_i$ with $\bar A$ by the above chosen biholomorphism). In particular, $\rho=1$ on the closed arc $L_\rho:=Z_\rho\cap\partial_o A_{ij}$.
 Then $G:=\frac{\partial(\rho F)}{\partial\bar z}$ is a bounded
$H^\infty(S_o')$-valued $C^\infty$ function on $A_{i}\cup Z_\rho$ with support in $\overline{A}_{ij}$ equals 0 on $Z_\rho$.  

Next, under the identification of $A_i$ with $A$ 
\begin{equation}\label{eq4.4}
H(\xi)=\frac{1}{\pi}\iint_{A_{ij}}\,\frac{G(z)}{\xi-z}\,dx\,dy,\qquad \xi\in A_{i}\cup Z_\rho,\quad x={\rm Re}\, z,\quad y={\rm Im}\, z,
\end{equation}
is a bounded $H^\infty(S_o')$-valued $C^\infty$ function on $A_{i}\cup Z_\rho$  holomorphic on $Z_\rho$ satisfying
\[
\frac{\partial H}{\partial\bar \xi}=G.
\]
Hence,
\[
H_1:=\rho F- H|_{A_i}\in H^\infty(A_i,H^\infty(S_o'))=H^\infty(A_i\times S_o').
\]
In turn, $F=H_1+H$ on $A_{ij}\cap Z_\rho$. Considering $H_1$ and $H$ as functions on $A_i\times S_o'$ and $(A_i\cup Z_\rho)\times S_o'$ and restricting them to $e(A_{ij}')$ we obtain
\[
f=h_1+h\quad {\rm on}\quad A_{ij}'\cap Z_\rho',\quad Z_\rho':=r^{-1}(Z_\rho);
\]
here $h_1:=e^*H_1\in H^\infty(A_i')$ and $h:=e^*H\in C(A_i'\cup Z_\rho')$ is such that $h|_{Z_\rho'}$ is bounded uniformly continuous  with respect to the path metric induced by a Riemannian metric pulled back by $r$ from $S_o$. According to Lemma \ref{lem4.3}, $h_1$ admits a continuous extension to $\hat r^{-1}(A_i\cup\partial D_i)$. Also, $h|_{Z_\rho'}$ admits a continuous extension to $r_E^{-1}(Z_\rho)\subset\mathfrak M(H^\infty(S_o'))$ and, hence, to $\hat r^{-1}( Z_\rho)$, see (*). These imply that $f|_{A_{ij}'\cap Z_\rho'}$ admits a continuous extension to $\hat r^{-1}((A_{ij}\cap Z_\rho)\cup L_\rho)$. Since
\[
\bigcup_\rho\, Z_\rho=A_{ij},
\]
where $\rho$ runs over all possible functions satisfying the above described conditions, the latter implies that
$f$ admits an extension $\tilde f\in C(\hat r^{-1}(A_{ij}\cup \partial_o A_{ij}))$, as required.
 \end{proof}
 
 Let $q_j:\mathfrak M(H^\infty(A_{ij}'))\rightarrow \mathfrak M(H^\infty(S'))$ be the continuous maps transposed to the embeddings $H^\infty(S')\hookrightarrow H^\infty(A_{ij}')$ determined by restrictions to $A_{ij}'$, $j=1,2$. Due to the corona theorem for $H^\infty(S')$ the image of $q_j$ coincides with $\hat r^{-1}(\overline{A}_{ij})$.
 Since $A_{ij}'$ is biholomorphic to $\Di\times\N$, the corona theorem is valid for $H^\infty(A_{ij}')$, i.e., $A_{ij}'$ is dense in $\mathfrak M(H^\infty(A_{ij}'))$. 
 
 As a consequence of Proposition \ref{prop4.2} we obtain the following:
\begin{C}\label{cor4.4}
The restriction of $q_j$ to $(\hat r\circ q_j)^{-1}(A_{ij}\cup \partial_o A_{ij})$ is a bijection onto the set $\hat r^{-1}(A_{ij}\cup \partial_o A_{ij})$. 
\end{C}
The result implies that $q_j$ maps a compact subset $K\subset (\hat r\circ q_j)^{-1}(A_{ij}\cup \partial_o A_{ij})$  bijectively onto the compact set $q_j(K)\subset \hat r^{-1}(A_{ij}\cup \partial_o A_{ij})$. Hence, $q_j|_K$ is a homeomorphism.
\begin{proof}
Since the mapping $q_j|_{(\hat r\circ q_j)^{-1}(A_{ij}\cup \partial_o A_{ij})}:(\hat r\circ q_j)^{-1}(A_{ij}\cup \partial_o A_{ij})\rightarrow \hat r^{-1}(A_{ij}\cup \partial_o A_{ij})$ is surjective, it remains to show that it is injective. 
 
 Suppose that  points $x_1,x_2\in (\hat r\circ q_j)^{-1}(A_{ij}\cup \partial_o A_{ij})$ are such that $q_j(x_1)=q_j(x_2)\in \hat r^{-1}(A_{ij}\cup \partial_o A_{ij})$. If $x_1\ne x_2$, then there is a function $f\in H^\infty(A_{ij}')$ whose Gelfand transform $\hat f\in C(\mathfrak M(H^\infty(A_{ij}')))$ satisfies $\hat f(x_1)\ne \hat f(x_2)$.  Let $\{x_{k\alpha}\}_{\alpha\in\Lambda}\subset A_{ij}'$ be nets converging to $x_k$ in $\mathfrak M(H^\infty(A_{ij}'))$, $k=1,2$.  Then, since $q_j|_{A_{ij}'}={\rm id}$, these nets converge to $q_j(x_k)$ in $\mathfrak M(H^\infty(S'))$, $k=1,2$.
 According to Proposition \ref{prop4.2}, $f$ admits an extension $\tilde f\in C(\hat r^{-1}(A_{ij}\cup \partial_o A_{ij}))$. Hence, due to the continuity of $\hat f$, $\tilde f$ and $q_j$ we obtain
 \[
 \begin{array}{r}
 \displaystyle
0\ne \hat f(x_2)-\hat f(x_1)=\lim_\alpha f(x_{2\alpha})-\lim_\alpha f(x_{1\alpha})=\lim_\alpha f(q_j(x_{2\alpha}))-\lim_\alpha f(q_j(x_{1\alpha}))\medskip\\
\displaystyle
=\tilde f(q_j(x_2))-\tilde f(q_j(x_1))=0,\qquad
\end{array}
 \]
 a contradiction proving the result.
 \end{proof}
 
 Let us complete the proof of the theorem. Recall that we require to show only that
${\rm dim}\,\hat r^{-1}(\partial S)\le 2$ (see the explanation after the proof of 
 Lemma \ref{lem4.1}).

By definition, $\hat r^{-1}(\partial S)=\sqcup_{i=1}^k\, \hat r^{-1}(\partial D_i)$ and each  $\hat r^{-1}(\partial D_i)$ is covered by two relatively open sets $\hat r^{-1}(\partial_o(A_{ij}))$, $j=1,2$. Due to Corollary \ref{cor4.4} and Theorem \ref{te2.1}(a) we have for each compact subset $K\subset \hat r^{-1}(\partial_o(A_{ij}))$ 
\[
{\rm dim}\,K={\rm dim}\,q_j^{-1}(K)\le {\rm dim}\,\mathfrak M(H^\infty(A_{ij}'))={\rm dim}\,\mathfrak M(H^\infty(\Di\times\N))=2.
\]
From here and \cite[Ch.2,\,Thm.9-11]{N} we obtain that ${\rm dim}\,\hat r^{-1}(\partial D_i)\le 2$ for all $i$. This implies that ${\rm dim}\,\hat r^{-1}(\partial S)\le \max_{1\le i\le k}\bigl\{ {\rm dim}\,\hat r^{-1}(\partial D_i)\bigr\}\le 2$, as required. 

Thus, we have proved that ${\rm dim}\,\mathfrak M(H^\infty(S'))\le 2$. In fact, since  $\mathfrak M(H^\infty(S'))$ contains the dense open subset $S'$ of dimension two,  ${\rm dim}\,\mathfrak M(H^\infty(S'))=2$ as well.

The proof of the theorem is complete.

\end{document}